\documentclass{amsart}

\usepackage{geometry,graphicx,amssymb,amsmath,amsbsy,eucal,amsfonts,mathrsfs,amscd,bm,color}
\usepackage{cite}
\usepackage[all]{xy}

\numberwithin{equation}{section}

\allowdisplaybreaks[4]

\newtheorem{theorem}{Theorem}[section]
\newtheorem{lemma}[theorem]{Lemma}
\newtheorem{corollary}[theorem]{Corollary}
\newtheorem{proposition}[theorem]{Proposition}
\theoremstyle{definition}

\theoremstyle{remark}

\newcommand{\revj}[1]{{{#1}}}

\newcommand{\dive}{{\ensuremath\mathop{\mathrm{div}\,}}}

\newcommand{\RRR}{\mathbb{R}}

\newcommand{\pol}{\EuScript{P}}

\newcommand{\bld}[1]{\boldsymbol{#1}}

\newcommand{\bv}{\bld{v}}

\newcommand{\bn}{\bld{n}}

\newcommand{\bV}{\bld{V}}

\newcommand{\bq}{\bld{q}}

\newcommand{\bsigma}{\bld{\sigma}}

\newcommand{\bx}{\bld{x}}

\newcommand{\Om}{\Omega}

\newcommand{\Oh}{{\mathcal{T}_h}}

\subjclass[2010]{Primary }

\title[Discrete extension operators]{Discrete Extension Operators for 
Mixed finite element spaces on locally refined meshes}
\author{Mark Ainsworth}
\email{mark\_ainsworth@brown.edu}
\address{Divison of Applied Mathematics, Brown University, Providence, RI}
\author{Johnny Guzm\'an}
\email{johnny\_guzman@brown.edu}
\address{Division of Applied Mathematics, Brown University, Providence, RI}
\author{Francisco-Javier Sayas}
\email{fjsayas@udel.edu}
\address{Department of Mathematical Sciences, University of Delaware, Newark,DE}
\thanks{Partial support for MA under AFOSR contract FA9550-12-1-0399 is
gratefully acknowledged. FJS was partially funded by the NSF grant DMS 1216356}\date{\today}

\keywords{finite elements, Stokes, conforming, divergence-free}
 
\subjclass{76M10,65N30,65N12}
 
\begin{document}
\maketitle

\begin{abstract}
The existence of uniformly bounded discrete extension operators is established
for conforming Raviart-Thomas and N\'edelec discretisations of $H(div)$ and
$H(curl)$ on locally refined partitions of a polyhedral domain into tetrahedra.
\end{abstract} 

\section{Introduction}
Many boundary value problems with non-homogeneous boundary data may be
cast in the abstract variational form: find $u\in X$ such that $\gamma u=g\in M$
and 
\begin{equation}\label{continuousB} 
 B(u,v) = F(v)\quad\forall v\in X_0
\end{equation}
where $X$ is a Hilbert space over a domain $\Omega$, $M$ is a Hilbert space
over the boundary $\Gamma$ of $\Omega$, and $\gamma:X\to M$ is a trace
operator with $X_0=\{v\in X:\gamma v=0\}$. We assume that the bilinear and
linear forms $B:X\times X\to\RRR$ and $F:X\to\RRR$ satisfy suitable conditions \revj{(e.g. inf-sup stability and continuity)}
for the problem to be well-posed; specific examples will be given later. 

A Galerkin finite element approximation $u_h\approx u$ is obtained by
selecting a finite dimensional subspace $X_h\subset X$, setting
$M_h=\gamma X_h$, constructing a suitable approximation $g_h\approx g$
of the non-homogeneous boundary data, and seeking $u_h\in X_h$ such that
$\gamma u_h=g_h\in M_h$ and 
\begin{equation}\label{discreteB}
 B(u_h, v_h) = F(v_h)\quad\forall v_h\in X_{h0}=\{v_h\in X_h: \gamma v_h=0\}.
\end{equation}
The discrete problem~\eqref{discreteB} is well-posed provided that there exists
a positive constant $\beta>0$ such that  
\begin{equation}\label{disinfsup}
\sup_{v_h\in X_{\revj{h0}}, \|v_h\|_X=1} B(w_h,v_h)\ge\beta\|w_h\|_X \quad \text{
for all } w_h \in X_{\revj{h0}}.
\end{equation}
The issue of the accuracy of the resulting approximation is usually
addressed by reference to the following classical C\'ea type estimate
\begin{equation}\label{Cea} 
 \|u-u_h\|_X \le \Big(1+\frac{\kappa}{\beta}\Big)
 \inf_{w_h\in X_h: \gamma w_h=g_h}\|u-w_h\|_X,
\end{equation}
where $\kappa$ is the continuity constant of the bilinear form $B$. \revj{To prove this result one first proves an estimate for $u_h-w_h$  and then applies the triangle inequality.   To do this, we note that $\gamma(u_h-w_h)=0$ then apply \eqref{disinfsup}, Galerkin orthogonality (e.g. \eqref{continuousB} and \eqref{discreteB}), and use continuity of $B$. As a side note, perhaps one can avoid using the triangle inequality by using the techniques in \cite{XuZikatanov}}.

 It is clear that the accuracy depends on both the choice of finite dimensional
subspace $X_h\subset X$ and on the choice $g_h\approx g$ of the approximate
Dirichlet boundary condition. Nevertheless, the bound~\eqref{Cea} is somewhat
unsatisfactory. In particular, whilst $X_h\subset X$ and $g_h\approx g$ can
essentially be chosen independently of one another, the influence of each choice
on the accuracy of the resulting finite element approximation is obscured
through the requirement that the choice of comparator $w_h\in X_h$ is
constrained to satisfy the boundary condition $\gamma w_h=g_h$. 

Under what conditions is it possible to obtain an error estimate of the form 
\begin{equation}\label{introestimate}
\|u-u_h\|_{X} \le C (\inf_{v_h\in X_h}  \|u-v_h\|_X + \|g-g_h \|_M),
\end{equation}
in which the individual contributions to the error corresponding to the
choice of $X_h\subset X$ and $g_h\approx g$ are isolated? Suppose that there
exists a uniformly bounded \emph{discrete extension} operator $L_h: M_h
\rightarrow X_h$ such that 
\begin{equation}
 (P1)\quad\gamma L_h \mu_h=\mu_h \quad\forall \mu_h \in M_h;
 \quad
 (P2)\quad\|L_h \mu_h\|_{X} \le C_L \|\mu_h\|_{M}\quad\forall \mu_h \in M_h.
\end{equation}
Let $v_h\in X_h$ be arbitrary, and set $w_h = v_h - L_h(\gamma v_h - g_h)\in
X_h$, so that $\gamma w_h = g_h$ on $\Gamma$ and $u_h-w_h\in X_{\revj{h0}}$.
With this choice, estimate~\eqref{Cea} then gives
\begin{equation}
 \|u - u_h\|_X \le \Big(1+\frac{\kappa}{\beta}\Big)\|u-v_h+L_h(\gamma v_h-g_h)\|_X.
\end{equation}
With the aid of the triangle inequality and $(P2)$,  the right hand side in
the above estimate may be bounded by $\|u-v_h\|_X + C_L\|\gamma v_h - g_h\|_M$.
The second term in this expression can be bounded by inserting $0=\gamma u -
g$, applying the triangle inequality and using the continuity of the trace
operator $\gamma:X\to M$ to obtain $\|\gamma v_h - g_h\|_M\le C\|v_h-u\|_X +
\|g-g_h\|_M$. Combining the above estimates, we conclude
that~\eqref{introestimate} holds whenever there exists a discrete extension
operator satisfying $(P1)$-$(P2)$. Interestingly, the existence
of an operator satisfying $(P1)-(P2)$ is also \emph{necessary}
for a bound of the form~\eqref{introestimate} to hold~\cite{DoSa:2003,
Sayas:2007}. 

The existence of uniformly bounded discrete extension operators satisfying
$(P1)$-$(P2)$ is important in many areas of numerical analysis including the
construction of domain decomposition preconditioners~\cite{Widlund}. The main
purpose of the current work is to establish the existence of discrete extension
operators satisfying $(P1)$-$(P2)$ in the case where the discrete spaces are
taken to be conforming discretisations of $H(\mathrm{div})$ and $H(\mathrm{curl})$, i.e.
Raviart-Thomas and N\'ed\'elec spaces, on \revj{on general shape regular mesh}
partitioning of a polyhedral domain into tetrahedra. 

Various results concerning stable extension operators are interspersed in the
literature. There is a number of results available in the literature
concerning discrete extensions from the boundary to the interior on a single
isolated element~\cite{DeGoSc:2008}, but the fact that the norms on the trace
spaces are not additive means that one cannot prove the results for collections
of elements by simply summing contributions from individual elements. The case
of Raviart-Thomas elements on a two-dimensional domain appears
in~\cite{BaGa:2003} applied to the analysis of weakly imposed essential boundary
conditions for the mixed Laplacian for the case of quasi-uniform triangulations,
and was subsequently extended~\cite{GaOySa:2012} to cover meshes that are
quasi-uniform in a neighborhood of the boundary. Subsequently, the case of
general non-quasi-uniform meshes was covered in~\cite{MaMeSa} although,
unfortunately, the arguments used seem to be limited to the two-dimensional
setting. \revj{Stable discrete extensions for $H(\text{curl}, \Omega)$ conforming spaces have also appeared in the literature and have important applications; see \cite{HiptmairJerez-HanckesMao, HiptmairMao, AlonsoValli}.  Again, the results in these articles assume some degree of quasi-uniformity.} 

The approach employed in the present work for the treatment of Raviart-Thomas
elements is similar to the idea used \revj{in~\cite[Lemma 3.2]{BaGa:2003} and \cite[Lemma 5.1]{GaOySa:2012}} without,
however, requiring quasi-uniformity of the mesh. The key to relaxing
the conditions on the mesh is to develop local regularity estimates along with
discrete norm equivalences valid on \revj{general shape regular meshes}~\cite{AiMcTr:1999},
and to show that the operator defined in \cite{GaOySa:2012} is, in fact,
uniformly bounded on general shape regular meshes. 
%The treatment of the N\'ed\'elec spaces is different again. 
\revj{ The proof of a bounded discrete extension operator for N\'ed\'elec  spaces is similar to a proof found in \cite{HiptmairMao}}.
The idea is to first split the discrete trace using a discrete Hodge decomposition and to then use our extension result
for Raviart-Thomas elements to handle one component of the splitting, with the
remaining component treated using an idea adapted from~\cite{MaMeSa}. The resulting extension yields a divergence-free field which therefore belongs Brezzi-Douglas-Marini space. Consequently, our results for the Raviart-Thomas
case (i.e. N\'ed\'elec spaces in the three dimensional case~\cite{Nedelec:1980})
extend to the three dimensional counterpart of the Brezzi-Douglas-Marini finite
element~\cite{Nedelec:1986,BrDoDuFo:1987}.

The plan of the paper is as follows. The main results are stated in Section 2.
Proofs are given in Sections 3 for the Raviart-Thomas elements, and in Section 4
for the N\'ed\'elec elements. The local regularity estimates needed in
Section 3 are given in the Appendix. Basic results on the spaces $H^1(\Omega)$,
$H(\text{div},\Omega)$, and $H(\text{curl},\Omega)$ will be assumed throughout.
The symbol $\lesssim$ will be used as follows: for two quantities $a_h$ and
$b_h$ depending on the triangulations (see below), we write $a_h\lesssim b_h$,
whenever there exists $C>0$ independent of $h$, such that $a_h \le C b_h$. The
quantity $C$ will be allowed to depend on: the polynomial degree, the
shape-regularity of the triangulation, and the domain.

\section{Main results}
Let $\Omega \subset \mathbb R^3$ be a connected polyhedral Lipschitz domain. The unit outward pointing normal vector field on $\Gamma:=\partial\Omega$ will be denoted by $\boldsymbol n$. Let now $\Oh$ be a shape regular simplicial triangulation of $\Omega$ which, however, need not be quasi-uniform. For each $K \in \Oh$ we let $h_K$ denote the diameter of $K$.  We then consider the spaces of Raviart-Thomas and N\'ed\'elec finite elements:
\begin{eqnarray*}
\bV_h &:=& \{ \bq \in H(\text{div},\Omega): \bq|_K \in [\pol^k(K)]^3+ \pol^k(K) \bx , \text{ for all } K \in \Oh\}  ,\\
\boldsymbol N_h &:=& \{ \bq\in H(\text{curl},\Omega):\bq|_K\in [\pol^k(K)]^3+[\pol^k(K)]^3 \times \bx,\text{ for all } K\in \Oh\},
\end{eqnarray*}
where $\pol_k(S)$ is the space of polynomials of degree $k$ or less defined on $S$.
On the boundary $\Gamma$, we consider the induced triangulation, 
\begin{equation*}
\Gamma_h=\{ \partial \Omega \cap \partial K: \text{ for all } K \in \Oh\},
\end{equation*}
and two spaces
\begin{eqnarray*}
M_h &:=& \{ \bq\cdot\boldsymbol n : \bq \in \boldsymbol V_h\} = \{ v: v|_F \in \pol^k(F) \text{ for all } F \in \Gamma_h \} \\
\boldsymbol R_h &:=& \{\bq\times\boldsymbol n : \bq\in \boldsymbol N_h\} 
=\{ \boldsymbol r \in H(\text{div}_\Gamma,\Gamma):  \boldsymbol r|_F \in [\pol^k(F)]^2+\pol^k(F)\boldsymbol{x}_t
\text{ for all } F \in \Gamma_h \}.
\end{eqnarray*}
In the last space $\boldsymbol{x}_t:=\boldsymbol x-(\boldsymbol x\cdot\boldsymbol n)\boldsymbol n$ is the tangential position vector, and $\text{div}_\Gamma$ is the tangential divergence operator. We also consider $M_h^0$ to be the subset of $M_h$ consisting of elements whose average value vanishes.

Let $S$ be a $d$-dimensional domain, then the fractional Sobolev norms on $S$ are defined as follows: for non-negative integer $k$ and $0<s<1$,
we define
\begin{equation*}
\|v\|_{H^{k+s}(S)}^2=\|v\|_{H^{k}(S)}^2+|v|_{H^{k+s}(S)}^2,
\end{equation*}
where
\begin{equation*}
|v|_{H^{k+s}(S)}^2 =\sum_{|\alpha|=k} \int_S \int_S \frac{|\partial ^\alpha v(x)-\partial^\alpha v(y)|^2}{|x-y|^{d+2s}} dx dy
\end{equation*}
is the Slobodetskij seminorm and $\|\,\cdot\,\|_{H^k(S)}$ is the usual Sobolev norm. For negative $s$, $H^{-s}(S)$ is the dual space of $H^s_0(S)$, the closure in $H^s(S)$ of the set of smooth compactly supported functions. In particular, in the case of the closed surface $\Gamma$, we can write for functions $v \in L^2(\Gamma)$:
\begin{equation*}
\|v\|_{H^{-1/2}(\Gamma)}=\sup_{w \in C^{\infty}(\Gamma), \|w\|_{H^{1/2}(\Gamma)}=1} \int_{\Gamma} v w. 
\end{equation*}

As noted before, the operator $\boldsymbol V_h\ni \boldsymbol v\mapsto \boldsymbol v\cdot\boldsymbol n\in M_h$ is surjective. The next result shows that there is a right-inverse of this operator that is  bounded as an operator $H^{-1/2}(\Gamma)\to H(\mathrm{div};\Omega)$, uniformly in the mesh size. Using the result  \cite[Theorem 5.1]{MaMeSa} it is enough to establish the uniform extension for data in $M_h^0$.

\begin{theorem}\label{the:MAIN1}
There exists a constant $C$ depending only on the shape regularity of $\Oh$ and on $\Omega$ such that for any $g_h \in M_h^0$ there exists $\bsigma_h \in \bV_h$ with the following properties: 
\begin{itemize}
\item[(a)] $\bsigma_h \cdot \bn=g_h$ \text{ on } $\Gamma$,
\item[(b)] $ \|\bsigma_h\|_{H(\mathrm{div};\Omega)} \lesssim   \|g_h\|_{H^{-1/2}(\partial \Omega)}$,
\item[(c)] $\mathrm{div}\, \boldsymbol\sigma_h =0$ \text{ in } $\Omega$. 
\end{itemize}
\end{theorem}

The second result concerns the N\'ed\'elec space, asserting the existence of a uniformly bounded right-inverse of the operator $\boldsymbol N_h\ni \boldsymbol w_h \mapsto \boldsymbol w_h\times \boldsymbol n\in \boldsymbol R_h$.

\revj{In order to state the  result we need to define the trace space of the $\mathrm{curl}$ operator.
\begin{equation*}
{\boldsymbol H}_{\|}^{-1/2} ({\rm div}_\Gamma; \Gamma):=\{ {\boldsymbol r} \in {\boldsymbol H}_{\|}^{-1/2}(\Gamma) : {\rm div}_\Gamma  {\boldsymbol r}\in H^{-1/2}(\Gamma) \},
\end{equation*}
with norm
\begin{equation*}
\|\boldsymbol r\|_{{\boldsymbol H}_{\|}^{-1/2} ({\rm div}_\Gamma; \Gamma)}^2 =\|\boldsymbol r\|_{{\boldsymbol H}_{\|}^{-1/2}( \Gamma)}^2+ \|{\rm div}_\Gamma  {\boldsymbol r}\|_{H^{-1/2}(\Gamma)}^2.
\end{equation*}
The space  ${\boldsymbol H}_{\|}^{-1/2}(\Gamma)$ along with the norm are defined in \cite{BuffaCiarlet1}; see also \cite{BuffaCiarlet2, BuffaCostabelSheen}. }

\revj{ In  \cite{BuffaCiarlet1} the following continuity result was proved. 
\begin{proposition}
 For $\boldsymbol w \in H(\mathrm{curl};\Omega)$, $\boldsymbol w  \times \boldsymbol n \in {\boldsymbol H}_{\|}^{-1/2} ({\rm div}_\Gamma; \Gamma)$ with the bound
\begin{equation}\label{CURLcontinuity}
\|\boldsymbol w \times \boldsymbol n \|_{{\boldsymbol H}_{\|}^{-1/2} ({\rm div}_\Gamma; \Gamma)} \lesssim  \|\boldsymbol w\|_{H (\mathrm {curl};\Omega)}.
\end{equation}
\end{proposition} }

\begin{theorem}\label{the:MAIN2}
There exists a constant $C$ depending only on the shape regularity of $\Oh$ and on $\Omega$ such that for any $\boldsymbol r_h \in \boldsymbol R_h$ there exists $\boldsymbol w_h \in \boldsymbol N_h$ with the following properties: 
\begin{itemize}
\item[(a)] $\boldsymbol w_h\times \boldsymbol n=\boldsymbol r_h$ \text{ on } $\Gamma$, 
\item[(b)] $ \|\boldsymbol w_h\|_{H(\mathrm{curl};\Omega)} \lesssim  \revj{\|\boldsymbol r_h\|_{{{\boldsymbol H}_{\|}^{-1/2} ({\rm div}_\Gamma; \Gamma)}}}$. 
\end{itemize}
\end{theorem}

\section{Discrete Extension Operators for Raviart-Thomas Finite Element Spaces}

We first recall some properties of the Raviart-Thomas projection. Let $\bv \in H^{1/2+s}(\Omega)$ for some $s>0$. Then we define $\Pi \bv \in \bV_h$ satisfying
\begin{alignat*}{6}
\int_{F} (\Pi \bv \cdot \bn) w&=\int_{F} (\bv \cdot \bn) w \qquad && \forall w\in \pol^k(F) \qquad& &\quad  \forall F\in \mathcal E_h,\\
\int_K \Pi\bv\cdot \boldsymbol w& =\int_K \bv\cdot\boldsymbol w & & \forall \boldsymbol w\in [\pol^{k-1}(K)]^3 & & \quad \forall K \in \Oh
\end{alignat*}
(see \cite[Example 2.5.3]{BoBrFo:2013}). Here $\mathcal E_h$ is the set of all faces of the triangulation. The following classical result can be found in \cite[Propositions 2.5.1, 2.5.2]{BoBrFo:2013}.

\begin{proposition}\label{prop:2.1}
For every  $\bv \in (H^{1/2+s}(\Omega))^3$ one has 
\begin{itemize}
\item[(a)] $\dive \Pi \bv=P \dive \bv$, 
\item[(b)] $\|\Pi \bv-\bv\|_{L^2(K)} \lesssim h_K^{1/2+s} \|\bv\|_{H^{1/2+s}(K)}$ for all $K \in \Oh$,
\end{itemize}
where  $P$ is the $L^2(\Omega)$-orthogonal projection onto the space of piecewise $\pol^k(K)$ functions.
\end{proposition}

The following inverse inequality will play a key role in our analysis.

\begin{lemma} \label{lemma:1}
For any $g \in M_h$ we have 
\begin{equation*}
\sum_{F \in \Gamma_h} h_F \|g_h \|_{L^2(F)}^2 \lesssim \|g_h\|_{H^{-1/2}(\Gamma)} ^2
\end{equation*} 
\end{lemma}

\begin{proof}
This result is a consequence  of basic estimates given in \cite{AiMcTr:1999}. Let $\{\varphi_{i,F}\,:\,i=1,\ldots,\mathrm{dim}\pol^k(F)\}$ be a basis for $\pol^k(F)$ built by pushing forward the Lagrange basis on the reference element. Then we decompose $g_h=\sum_F \sum_i g_{i,F}\varphi_{i,F}$, and estimate
\begin{alignat*}{6}
\sum_{F\in \Gamma_h} h_F \| g_h\|_{L^2(F)}^2 
	\lesssim & \sum_F \sum_i \revj{h_F} |g_{i,F}|^2 \| \varphi_{i,F}\|_{L^2(F)}^2\\
	\lesssim & \sum_F \sum_i h_F^3  |g_{i,F}|^2 & \qquad 
			& \mbox{(simple computation)}\\
	\lesssim & \sum_F \sum_i \| g_{i,F}\varphi_{i,F}\|_{H^{-1/2}(\Gamma)}^2 & 
			& \mbox{(by \cite[Theorem 4.8]{AiMcTr:1999})}\\
	\lesssim & \| g_h\|^2_{H^{-1/2}(\Gamma)}, &
			&\mbox{(by \cite[Lemma 5.4]{AiMcTr:1999})}
\end{alignat*}
which finishes the proof.
\end{proof}

We also need elliptic regularity results; see for example \cite{Dauge} for the case $g=0$. Consider the Poisson problem with Neumann boundary conditions
\begin{subequations}\label{eq:poisson}
\begin{eqnarray}
-\triangle u=&f \qquad \text{ on } \Omega \\
\nabla u \cdot \bn =&g \qquad  \text{ on } \Gamma,
\end{eqnarray}
\end{subequations}
under the assumption that $\int_\Gamma g+\int_\Omega f=0$ and $\int_{\Omega} u=0$. Then, there exist $C>0$ and  $s\in (0,1/2)$ such that
\begin{equation}\label{eq:reg}
\|u\|_{H^{3/2+s}(\Omega)} \le  C\big( \|f\|_{H^{-1/2+s}(\Omega)}+\|g\|_{H^s(\Gamma)}\big).
\end{equation}
We can localize this regularity result to obtain:

\begin{theorem}\label{thm1}
Suppose that $f \equiv  0$ in \eqref{eq:poisson} and let $\bsigma=\nabla u$. Then for each $K \in \Oh$ we have 
\begin{equation*}
\|\bsigma\|_{H^{1/2+s}(K)} \le C ( h_K^{-1/2-s} \| \bsigma\|_{L^2(D_K)} + \|g\|_{H^s(\partial D_K \cap \Gamma)} + \,h_K^{-s}\|g\|_{L^2(\partial D_K \cap \Gamma)}  ),
\end{equation*}
where
\[
D_K:=\cup\{ K'\in \Oh\,:\, \overline K\cap \overline{K'} \neq\emptyset\}
\]
is the collection of tetrahedra sharing  one or more vertices with $K$.
\end{theorem}
The proof of this result is contained in Appendix \ref{sectionlocal}.

\begin{proof}[Proof of Theorem \ref{the:MAIN1}]
Let $u$  satisfy
\begin{eqnarray*}
-\triangle u=&0 \qquad \text{ on } \Omega, \\
\nabla u \cdot \bn =&g_h \qquad  \text{ on } \Gamma,
\end{eqnarray*}
and set $\bsigma=\nabla u\in (H^{1/2+s}(\Omega))^3$ (see \eqref{eq:reg}). Note that $\dive \bsigma=0$. We define $\bsigma_h =\Pi \bsigma$ and we note that Theorem \ref{the:MAIN1} (a) holds, and that by Proposition \ref{prop:2.1}(a), $\mathrm{div}\,\boldsymbol\sigma_h=0$. Therefore, using elliptic regularity, Proposition \ref{prop:2.1}(b) and Lemma \ref{lemma:1}, 
\begin{alignat*}{6}
\|\bsigma_h\|_{H(\text{div};\Omega)} =&\|\bsigma_h\|_{L^2(\Omega)}\le \|\bsigma\|_{L^2(\Omega)}+\|\bsigma- \bsigma_h\|_{L^2(\Omega)} \\
	\lesssim &  \|g_h\|_{H^{-1/2}(\Gamma)}+\|\bsigma- \bsigma_h\|_{L^2(\Omega)}\\
	\lesssim & \|g_h\|_{H^{-1/2}(\Gamma)}+(\sum_{K \in \Oh} h_K^{2(1/2+s)} \|\bsigma\|_{H^{1/2+s}(K)}^2)^{1/2}\\    
	\lesssim &  \|g_h\|_{H^{-1/2}(\Gamma)} \\
		  & +\left(\sum_{K \in \Oh} \|\bsigma\|_{{L^2}(D_K)}^2
			 +h_K^{2(1/2+s)} \|g_h\|_{H^{s}(\partial D_K \cap \Gamma)}^2
			+h_K \|g_h\|_{L^2(\partial D_K \cap \Gamma)}^2\right)^{\!\!1/2} .
\end{alignat*}
The shape regularity of the elements means that
\begin{equation*}
\sum_{K \in \Oh} \|\bsigma\|_{L^2(D_K)}^2 \lesssim \|\bsigma\|_{L^2(\Omega)}^2.
\end{equation*}
Also, if we let $D_F$ to be the macro-element surrounding $F$ (triangles sharing a vertex with $F$), we have that
\begin{alignat*}{6}
\sum_{K \in \Oh} h_K^{2(1/2+s)} \|g_h\|_{H^{s}(\partial D_K \cap \Gamma)}^2 
	\lesssim  & \sum_{F\in  \Gamma_h} h_F^{2(1/2+s)} \|g_h\|_{H^{s}(D_F)}^2\\
	\lesssim &  \sum_{F\in  \Gamma_h} h_F \|g_h\|_{L^2(D_F)}^2 
	\lesssim  \sum_{F\in \Gamma_h} h_F \|g_h\|_{L^2(F)}^2,
\end{alignat*}
where we have used a standard local inverse estimate for piecewise polynomial functions. 
Applying Lemma \ref{lemma:1} completes the proof.
\end{proof}

As an application we can get an error estimate for the Laplacian in mixed form with Neumann boundary conditions. In this case $X= H(\text{div}, \Omega) \times L^2(\Omega)$  and the trace space is $M= H^{-1/2}(\partial \Omega)$ with trace operator $\gamma (\boldsymbol \sigma, u)= \boldsymbol  \sigma \cdot \boldsymbol n.$ The bilinear form $B$ and the linear form $F$ are given by
\begin{equation*}
B(( \boldsymbol \sigma, u), (\boldsymbol  \eta, w)):=\int_{\Omega} (\boldsymbol  \sigma \cdot \eta - u \, \mathrm{div}\, \boldsymbol  \eta + w\, \mathrm{div}\, \boldsymbol  \sigma) \,  
\end{equation*}
\begin{equation*}
F((\boldsymbol  \eta,w)):= \int_{\Omega} f w 
\end{equation*}
for a given $f \in L^2$.
Of course, the finite element space will be $X_h= \boldsymbol   V_h \times U_h$ where
\begin{equation*}
U_h  =  \{ w:\Omega\to\mathbb R\,:\, w|_K\in \pol^k(K)\quad \forall K \in \Oh\}.
\end{equation*}

In view of Theorem \ref{the:MAIN1} and the introductory discussion, we have the
following error estimates for Raviart-Thomas elements. 
\begin{corollary}
Let $X= H(\mathrm{div}, \Omega) \times L^2(\Omega)$ , $M= H^{-1/2}(\partial \Omega)$ and $X_h= \boldsymbol   V_h \times U_h$. Let $(\sigma, u)$ satisfy \eqref{continuousB} and $(\sigma_h, u_h)$ satisfy \eqref{discreteB} then the following holds 
\[
\|\boldsymbol\sigma-\boldsymbol\sigma_h\|_{\boldsymbol H(\mathrm{div};\Omega)}
+\|u-u_h\|_{L^2(\Omega)}
\lesssim \inf_{\boldsymbol v\in \boldsymbol V_h}\|\boldsymbol \sigma-\boldsymbol v\|_{\boldsymbol H(\mathrm{div};\Omega)}
+\inf_{w\in U_h} \| u-w\|_{L^2(\Omega)}
+\| g-g_h\|_{H^{-1/2}(\Gamma)}.
\]
\end{corollary}

Finally, we note that other applications of the existence of an extension
operator like $ L_h$ are given in the analysis of a variety of discretization
methods for the Stokes-Darcy problem~\cite{GaOySa:2012,MaMeSa}.

\section{Discrete Extension Operators for N\'ed\'elec Finite Element Spaces}

The proof of Theorem \ref{the:MAIN2} relies on a Helmholtz-Hodge
type decomposition of $\boldsymbol R_h$, two liftings (one for Lagrange finite
elements and the one provided by Theorem \ref{the:MAIN1}), and local estimates
in the space of Lagrange finite elements on the boundary
\[
P_h:=\{ \phi_h \in \mathcal C(\Gamma)\,:\,\phi_h|_F \in \pol_{k+1}(F) \text{ for all } F\in \Gamma_h\}.
\]
We begin with some technical results:
\begin{lemma}\label{lemma:44.1}
If $\boldsymbol r_h \in \boldsymbol R_h$ satisfies $\mathrm{div}_\Gamma \boldsymbol r_h=0$, then $\boldsymbol r_h =\mathrm{curl}_\Gamma \phi_h$ with $\phi_h \in P_h$.
\end{lemma}

\begin{proof}
By definition of $\boldsymbol R_h$, there exists $\boldsymbol w_h \in \boldsymbol N_h$ such that $\boldsymbol w_h\times \boldsymbol n=\boldsymbol r_h$. Note now that
\[
\boldsymbol q_h:=\nabla \times \boldsymbol w_h \in \boldsymbol V_h, \qquad \boldsymbol q_h\cdot\boldsymbol n=\mathrm{div}_\Gamma (\boldsymbol w_h \times \boldsymbol n)=0.
\]
Therefore, by the exactness of the discrete de-Rham complex with essential boundary conditions (see for example \cite{AFW}), there exists $\boldsymbol w_h^0\in \boldsymbol N_h$ such that $\nabla \times \boldsymbol w_h^0=\boldsymbol q_h$ and $\boldsymbol w_h^0\times \boldsymbol n=0$. We then consider the difference $\boldsymbol w_h-\boldsymbol w_h^0$ and note, again, by the exactness of the discrete de-Rham complex, there exists $u_h$ in the finite element space
\begin{equation}\label{eq:44.0}
W_h:=\{ u_h \in \mathcal C(\Omega)\,:\, u_h|_K\in \pol_{k+1}(K) \text{ for all } K\in \Oh\},
\end{equation}
satisfying $\nabla u_h=\boldsymbol w_h-\boldsymbol w_h^0$. Take  $\phi_h=u_h|_\Gamma$. The result follows since $\boldsymbol w_h\times \boldsymbol n=\nabla u_h\times \boldsymbol n=\mathrm{curl}_\Gamma \phi_h$.
\end{proof}

We will also need the following result found for example in \cite{BuffaCostabelSheen}.
\begin{lemma}\label{lemma:44.2}
For all $\phi \in H^{1/2}(\Gamma)$ we have 
\begin{equation*}
\|\phi\|_{H^{1/2}(\Gamma)/{\mathbb R}} \lesssim \| \mathrm{curl}_\Gamma \phi\|_{H_{\|}^{-1/2}(\Gamma)}.
\end{equation*}
\end{lemma}

\begin{lemma}\label{lemma:44.3}
For all $\boldsymbol v_h\in \boldsymbol V_h$ with $\mathrm{div}\, \boldsymbol v_h=0$, there exists $\boldsymbol w_h\in \boldsymbol N_h$ such that $\nabla\times \boldsymbol w_h=\boldsymbol v_h$ and $\| \boldsymbol w_h\|_{L^2(\Omega)} \lesssim \| \boldsymbol v_h\|_{L^2(\Omega)}.$
\end{lemma}

\begin{proof}
By \cite{ChWi:2008} uniformly bounded projections $\boldsymbol\Pi_h:H(\mathrm{curl},\Omega)\to \boldsymbol N_h$ and $\boldsymbol Q_h:H(\mathrm{div},\Omega)\to\boldsymbol V_h$ exist such that $\nabla\times \boldsymbol\Pi_h \boldsymbol w=\boldsymbol Q_h(\nabla\times \boldsymbol w)$ for all $\boldsymbol w\in H(\mathrm{curl},\Omega)$. On the other hand the curl operator is surjective from $H(\mathrm{curl},\Omega)$ to $\{ \boldsymbol v\in  H(\mathrm{div},\Omega) \,:\, \mathrm{div}\boldsymbol v=0\}$ and has therefore a bounded right-inverse. We then apply this right-inverse to $\boldsymbol v_h$ to obtain $\boldsymbol w\in H(\mathrm{curl},\Omega)$ and define $\boldsymbol w_h=\boldsymbol\Pi_h\boldsymbol w$. Then $\nabla \times\boldsymbol w_h=\boldsymbol Q_h(\nabla\times \boldsymbol w)=\boldsymbol Q_h\boldsymbol v_h=\boldsymbol v_h$ and the result is proved.
\end{proof}

\begin{proof}[Proof of Theorem \ref{the:MAIN2}]
Let $\boldsymbol r_h \in \boldsymbol R_h$. Then $\mathrm{div}_\Gamma \boldsymbol r_h\in M_h^0$ and by Theorem \ref{the:MAIN1} we can find $\boldsymbol v_h \in \boldsymbol V_h$ such that
\[
\mathrm{div}\,\boldsymbol v_h =0, \qquad \boldsymbol v_h\cdot\boldsymbol n=\mathrm{div}_\Gamma \boldsymbol r_h,
\]
and 
\[
\| \boldsymbol v_h\|_{H(\mathrm{div},\Omega)}\lesssim \| \mathrm{div}\,\boldsymbol r_h\|_{H^{-1/2}(\Gamma)}.
\]
We then use Lemma \ref{lemma:44.3} to obtain $\boldsymbol w_h\in \boldsymbol N_h$ such that $\nabla \times \boldsymbol w_h =\boldsymbol v_h$, and 
\begin{equation}\label{eq:44.1}
\| \boldsymbol w_h\|_{H(\mathrm{curl},\Omega)}\lesssim
\| \boldsymbol v_h\|_{H(\mathrm{div},\Omega)}\lesssim \| \mathrm{div}\,\boldsymbol r_h\|_{H^{-1/2}(\Gamma)}.
\end{equation}
 Consider now the function $\boldsymbol m_h:=\boldsymbol r_h-\boldsymbol w_h\times \boldsymbol n \in \boldsymbol R_h$ and note that
\[
\|\boldsymbol m_h\|_{\revj{ \boldsymbol H_{\|}}^{-1/2}(\Gamma)}\lesssim \|\boldsymbol r_h\|_{\revj{ \boldsymbol H_{\|}}^{-1/2}(\Gamma)}+
\|\mathrm{div}_\Gamma\boldsymbol r_h\|_{H^{-1/2}(\Gamma)}
\]
by \eqref{eq:44.1} and the continuity of the tangential trace operator from $H(\mathrm{curl},\Omega)$ to $\revj{{\boldsymbol H}_{\|}^{-1/2} ({\rm div}_\Gamma; \Gamma)}$. Additionally
\[
\mathrm{div}_\Gamma \boldsymbol m_h=\mathrm{div}_\Gamma\boldsymbol r_h-(\nabla \times \boldsymbol w_h)\cdot\boldsymbol n=0,
\]
by construction of $\boldsymbol w_h$. We then apply Lemma \ref{lemma:44.1} to find $\phi_h \in P_h$ such that
$\boldsymbol r_h-\boldsymbol w_h\times \boldsymbol n=\mathrm{curl}_\Gamma\phi_h$ and and use Lemma \ref{lemma:44.2} to bound

\begin{equation}\label{eq:44.3}
\| \phi_h\|_{H^{1/2}(\Gamma)\revj{/{\mathbb R}}} \lesssim \|\boldsymbol r_h-\boldsymbol w_h\times\boldsymbol n\|_{\boldsymbol H_{\revj{\|}}^{-1/2}(\Gamma)}
\lesssim \|\boldsymbol r_h\|_{\boldsymbol H_{\revj{\|}}^{-1/2}(\Gamma)}+
\|\mathrm{div}_\Gamma\boldsymbol r_h\|_{H^{-1/2}(\Gamma)},
\end{equation}
\revj{ where  used \eqref{CURLcontinuity} and \eqref{eq:44.1}}.

We then take $u_h$ in the finite element space $W_h$ (see \eqref{eq:44.0}) such that $u_h|_\Gamma=\phi_h$ and 
\begin{equation}\label{eq:44.4}
\| u_h\|_{H^1(\Omega)\revj{/ \mathbb{R}}}\lesssim \| \phi_h\|_{H^{1/2}(\Gamma) \revj{/\mathbb{R}}}. 
\end{equation}
This can be accomplished by first taking $u\in H^1(\Omega)$ whose trace is $\phi_h$ and satisfying $\|u\|_{H^1(\Omega)\revj{/\mathbb{R}}}\lesssim \|\phi_h\|_{H^{1/2}(\Gamma)\revj{/\mathbb{R}}}$ and then applying the Scott-Zhang interpolation operator \cite{ScZh:1990} to $u$. The desired lifting of $\boldsymbol r_h$ is the function $\boldsymbol w_h+\nabla u_h \in N_h$. The bound
\[
\| \boldsymbol w_h+\nabla u_h\|_{H(\mathrm{curl},\Omega)}\le 
\|\boldsymbol w_h\|_{H(\mathrm{curl},\Omega)}+\|\nabla u_h\|_{L^2(\Omega)}\lesssim 
\|\boldsymbol r_h\|_{\revj{\boldsymbol H_{\|}}^{-1/2}(\Gamma)}+
\|\mathrm{div}_\Gamma\boldsymbol r_h\|_{H^{-1/2}(\Gamma)}
\]
is a direct consequence of \eqref{eq:44.1}, \eqref{eq:44.3}, and \eqref{eq:44.4}. The fact that it is a lifting follows from
\[
(\boldsymbol w_h+\nabla u_h)\times \boldsymbol n=\boldsymbol w_h\times \boldsymbol n+\mathrm{curl}_\Gamma \phi_h=
\boldsymbol w_h\times \boldsymbol n+\boldsymbol m_h=
\boldsymbol r_h.
\]
This finishes the proof.
\end{proof}

As an application to Theorem \ref{the:MAIN2} we consider problem \eqref{continuousB} with
\[
X=H(\text{curl}; \Omega),
\qquad
M=\boldsymbol H_{\revj{ \|}}^{-1/2}(\mathrm{div}_{\Gamma}; \partial \Omega):=\{ \mu \in \boldsymbol H_{\revj{\|}}^{-1/2}(\partial \Omega): \mathrm{div}_\Gamma \mu \in H^{-1/2}(\partial \Omega) \}. 
\]
The finite element space are the N\'ed\'elec elements $X_h=\boldsymbol N_h$. The bilinear form $B$ and linear form $F$ are as follows 
\begin{equation*}
B(\boldsymbol u, \boldsymbol v)=\int_\Omega (\nabla \times \boldsymbol u)\cdot(\nabla \times \boldsymbol v)+\boldsymbol u\cdot\boldsymbol v,
\end{equation*}
\begin{equation*}
F(\boldsymbol v)= \int_\Omega \boldsymbol f \cdot\boldsymbol v. 
\end{equation*}

Theorem \ref{the:MAIN2} and the Introductory discussion now gives the 
following error estimates for N\'ed\'elec elements. 
\begin{corollary}
Let $X= H(\mathrm{curl}; \Omega)$ , $M=\boldsymbol H_{\revj{\|}}^{-1/2}(\mathrm{div}_{\Gamma}; \partial \Omega)$ and $X_h= \boldsymbol N_h$. Let $\boldsymbol u$ satisfy \eqref{continuousB} and $\boldsymbol u_h$ satisfy \eqref{discreteB} then the following holds 
\[
\| \boldsymbol u-\boldsymbol u_h\|_{H(\mathrm{curl};\Omega)}\lesssim \inf_{\boldsymbol v \in \boldsymbol N_h} \| \boldsymbol u-\boldsymbol v\|_{H(\mathrm{curl};\Omega)} + \| \boldsymbol g-\boldsymbol g_h||_{\boldsymbol H_{\revj{\|}}^{-1/2}(\mathrm{div}_\Gamma;\Gamma)}.
\]
\end{corollary}

The 

\appendix

\section{Proof of Theorem \ref{thm1}}\label{sectionlocal}
For each $K \in \Oh$ we can find a cut-off function $\omega=\omega_K \in C^\infty(D_K)$ with the following properties:
\begin{subequations}
\begin{alignat}{6}
\label{eq:om1}
& \omega \equiv 1 \mbox{ in $K$},\\ 
\label{eq:om2}
& \omega\equiv 0 \mbox{ in $\Omega\setminus D_K$},\\
\label{eq:om3}
& \|D^s \omega\|_{L^{\infty}(D_K)} \lesssim h_K^{-s} \mbox{ for $s=0,1,2.$}
\end{alignat}
\end{subequations}  
Note that
\begin{equation}\label{eq:4.1}
\|u\|_{H^{3/2+s}(K)} 
	\le  \|\omega u\|_{H^{3/2+s}(\Om)}
	\le  C (\|-\triangle(\omega u)\|_{H^{-1/2+s}(\Omega)} + \|\nabla (\omega u) \cdot \bn\|_{H^s(\Gamma)}),
\end{equation}
by \eqref{eq:reg}. \revj{Here the constant only depends on $\Omega$}.

Let us first deal with elements $K$ such that $D_K$ does not contain a face in $\partial \Omega$.  In this case $\nabla (\omega u) \cdot \bn\equiv 0$ on $\partial \Omega$. To bound the first term we let $v \in H^{1/2-s}(\Omega)$, and define $m(v)=\frac{1}{|D_K|}\int_{D_K} v$.
Then, we have 
\begin{eqnarray*}
-\int_{\Omega}\triangle(\omega u) v=-\int_{\Omega} \triangle(\omega u) (v-m(v))-\int_{\Omega} \triangle(\omega u) m(v) 
                                   =-\int_{\Omega} (2\nabla\omega \cdot \nabla u+u\triangle \omega) (v-m(v)) 
\end{eqnarray*}
where we used that $\triangle u=0$ and that $\nabla (\omega u) \cdot \bn\equiv 0$ on $\partial \Omega$.  
Therefore
\begin{eqnarray*}
-\int_{\Omega}\triangle(\omega u) v &\le & (\|\nabla\omega\cdot \nabla u\|_{L^2(D_K)} +\|u\triangle \omega\|_{L^2(D_T)}) \|v-m(v)\|_{L^2(D_K)} \\
&\lesssim  & (h_T^{-1/2-s}\|\nabla u\|_{L^2(D_K)} +h_K^{-3/2-s}\|u\|_{L^2(D_K)}) \|v\|_{H^{1/2-s}(D_K)},
\end{eqnarray*}
where we have used \eqref{eq:om3}, the Poincar\'e inequality and an interpolation argument.
Taking the supremum over $v$ we have 
\begin{equation*}
\|-\triangle(\omega u)\|_{H^{-1/2+s}(\Omega)} \le (h_K^{-1/2-s}\|\nabla u\|_{L^2(D_K)} +h_K^{-3/2-s}\|u\|_{L^2(D_K)}). 
\end{equation*}
Hence, in the case $D_K$ does not contain a face in $\partial \Omega$ we have 
\begin{equation*}
\|u\|_{H^{3/2+s}(K)} \lesssim (h_K^{-1/2-s}\|\nabla u\|_{L^2(D_K)} +h_K^{-3/2-s}\|u\|_{L^2(D_K)})
\end{equation*} 
If we replace $u$ with $m(u)$, 
and note that
\[
\|u-m(u)\|_{L^2(D_K)} \lesssim h_K \|\nabla u\|_{L^2(D_K)}
\]
we get
\begin{equation*}
\|\nabla u \|_{H^{1/2+s}(K)} \lesssim h_K^{-1/2-s}\|\nabla u\|_{L^2(D_K)}.
\end{equation*} 

Next we consider the case when $\partial D_K$ contains one or more faces on $\Gamma$. To bound the first term in the right of \eqref{eq:4.1} we get 
\begin{eqnarray*}
\left|\int_{\Omega}\triangle(\omega u) v \right|
&\le & (2\|\nabla \omega  \cdot \nabla u\|_{L^2(D_K)} +\|u\triangle \omega\|_{L^2(D_T)}) \|v\|_{L^2(D_K)} \\
&\lesssim &  (h_K^{-1/2-s}\|\nabla u\|_{L^2(D_K)} +h_K^{-3/2-s}\|u\|_{L^2(D_K)}) h_K^{1/2-s} \|v\|_{H^{1/2-s}(D_K)},
\end{eqnarray*}
where we have used \eqref{eq:om3} and the fact that $v$ vanishes on at least one face of $\partial D_K$, which allows us to use an inequality in the form
\begin{equation}\label{eq:Fr}
\|v\|_{L^2(D_K)} \lesssim h_K^{1/2-s} \|v\|_{H^{1/2-s}(D_K)}. 
\end{equation}

Therefore, as above we have that 
\begin{equation*}
\|-\triangle(\omega u)\|_{H^{-1/2+s}(\Omega)}\lesssim h_K^{-1/2-s}\|\nabla u\|_{L^2(D_K)} +h_K^{-3/2-s}\|u\|_{L^2(D_K)}.
\end{equation*}
To bound the second term on the right of \eqref{eq:4.1} we first use  the product rule and get
\begin{eqnarray*}
 \|\nabla (\omega u) \cdot \bn\|_{H^s(\partial \Omega)} &\le&  \|u \nabla \omega \cdot \bn\|_{H^s(\partial \Omega)}+ \|\omega g\|_{H^s(\Gamma)}\\
& \lesssim & h_K^{-1/2-s}\|u \nabla \omega\|_{L^2(D_T)}+h_K^{1/2-s} \|\nabla (u \nabla \omega)\|_{L^2(D_T)} + \|\omega g\|_{H^s(\Gamma)}\\
&\lesssim &  h_K^{-3/2-s}\|u \|_{L^2(D_T)}+h_K^{-1/2-s} \|\nabla u \|_{L^2(D_T)}+\|\omega g\|_{H^s(\Gamma)},
\end{eqnarray*}
after using a localized version of the trace theorem and \eqref{eq:om3}. By a simple interpolation argument and \eqref{eq:om3}, we have
\begin{equation*}
\|\omega g\|_{H^s(\Gamma)} \lesssim \|g\|_{H^s(\partial D_K \cap \Gamma)} + h_K^{-s} \|g\|_{L^2(\partial D_K\cap \Gamma)}.
\end{equation*}
Combining the above inequalities we have 
\begin{eqnarray*}
\|u\|_{H^{3/2+s}(K)} \lesssim & (h_K^{-1/2-s}\|\nabla u\|_{L^2(D_K)} +h_K^{-3/2-s}\|u\|_{L^2(D_K)})\\
&+(\|g\|_{H^s(D_K \cap \Gamma)} + h_K^{-s} \|g\|_{L^2(D_K\cap \Gamma)}).
\end{eqnarray*}
If we apply the above argument to $u-m(u) $, we obtain our result.

\revj{{\bf Acknowledgements}: The authors would like to thank Norbert Heuer for several useful discussion. Also, we thank  Ralf Hiptmair for bringing  the relevant papers \cite{AlonsoValli, HiptmairMao, HiptmairJerez-HanckesMao} to our attention.}

%%
% requires a BiBTeX file sample.bib
\bibliographystyle{plain}

%\bibliography{listofpapers}

\end{document}